\documentclass[12pt]{amsart}
\usepackage{graphicx} 
\usepackage{amsmath}
\usepackage{amssymb}
\usepackage{amsfonts}
\usepackage{mathtools}
\usepackage{amsthm}  
\usepackage{dsfont}
\usepackage[colorlinks]{hyperref}
\usepackage{mathrsfs}
\usepackage{anysize}
\marginsize{1.5cm}{1.5cm}{2cm}{2cm}

\usepackage{bm}

\newcommand{\R}{\mathbb{R}}
\newcommand{\C}{\mathbb{C}}

\theoremstyle{plain}
\newtheorem{theorem}{Theorem}[section]

\newtheorem*{theorem*}{Theorem}
\newtheorem{proposition}[theorem]{Proposition}
\newtheorem{corollary}[theorem]{Corollary}
\newtheorem{lemma}[theorem]{Lemma}

\theoremstyle{definition} 
\newtheorem{definition}[theorem]{Definition} 
\newtheorem{remark}[theorem]{Remark} 

\numberwithin{equation}{section}

\title[\texorpdfstring{Analytic Extensions of $A_{\infty}$-weights}{}]{Analytic extensions of $A_{\infty}$-weights on Lipschitz curves and their use in weighted Hardy spaces}

\author[F. Ballesta-Yag\"ue]{Fernando Ballesta-Yag\"ue}
\address[F. Ballesta-Yag\"ue]{Departamento de An\'alisis Matem\'atico y Matem\'atica Aplicada, Facultad de Ciencias Matem\'aticas, Universidad Complutense de Madrid \hfill\break\indent Pl. de las Ciencias 3, 28040 Madrid, Spain}
\email{ferballe@ucm.es}

\allowdisplaybreaks
\begin{document}

\begin{abstract}

    An $A_{\infty}$-weight on a Lipschitz curve $\Lambda$ in the plane can be extended analytically to the graph Lipschitz domain $\Omega$ above it. This problem was studied by C. Kenig~\cite{kenigWeighted}, who introduced the class $AE$ of well-behaved analytic extensions. Later, he and D. Jerison~\cite{jerisonKenig} added a Smirnov-type condition to the definition of this class.  

    In this note, we show that this Smirnov-type condition is equivalent to an $H^1$-integrability condition. As a consequence, one of the conditions in the definition of $AE$ can be dropped.  We use this simplification to apply C. Kenig's theory to prove results about weighted Hardy spaces. These are useful to study the Neumann problem in $\Omega$ with boundary data in weighted spaces.

\end{abstract}

\thanks{
The author was supported by grants PID2020-113048GB-I00 funded by MCIN/AEI/10.13039/501100011033,  CEX2019-000904-S funded by MCIN/AEI/ 10.13039/501100011033 and Grupo UCM-970966 (Spain), and benefited from an FPU Grant FPU21/06111 from Ministerio de Universidades (Spain).
}


\maketitle

\section{Introduction}

Let $\Omega$ be a graph Lipschitz domain and let $\nu$ be a measure over $\Lambda$, the boundary of $\Omega$. The class $AE(\Omega,\nu)$, where $AE$ stands for \textit{analytic extensions}, is a beautiful construction that allows us to deal with the following problem. In weighted $L^p$ spaces, the weight can be moved from the measure to the function easily in the following sense: if $v$ and $w$ are weights, then
\[
\|fw\|_{L^p(v)}
=
\|f\|_{L^p(vw^{p})}.
\]
However, the same cannot be done with such ease for Hardy spaces, since the definition of the $H^p(\Omega,d\nu)$-norm 
\[
\|F\|_{H^p(\Omega,d\nu)}
\simeq_{\alpha}
\|\mathcal{M}_{\alpha}(F)\|_{L^p(\Lambda,d\nu)}
\]
in terms of the non-tangential maximal operator $\mathcal{M}_{\alpha}$ involves a supremum. Therefore, given $\nu$ and $\mu$ two measures on $\Lambda$ absolutely continuous with respect to arc-length measure $ds$, and whose density functions are denoted
by $\nu$ and $\mu$ as well, it is not clear if the following chain of equivalences holds:
\begin{equation}\label{eq:desiredChangeMeasureHpNorm}
\|F\|_{H^p(\Omega,\nu^{-p} \mu ds)}
\simeq_{\alpha}
\left\|
\frac{\mathcal{M}_{\alpha}(F)}{\nu}
\right\|_{L^p(\Lambda,d\mu)}
\overset{(\text{?})}{\simeq}
\left\|
\mathcal{M}_{\alpha}\left(\frac{F}{\nu}\right)
\right\|_{L^p(\Lambda,d\mu)}
\simeq_{\alpha}
\left\|
\frac{F}{\nu}
\right\|_{H^p(\Omega,d\mu)}.
\end{equation}
In order for the above to even make sense, we would need $\frac{F}{\nu}$ to be an analytic function on $\Omega$. But $\nu$ is a function defined on $\Lambda$. Therefore, to begin with, we would need to obtain some kind of \textit{analytic extension} of the weight $\nu$, namely $G$, from the boundary $\Lambda=\partial\Omega$ to the whole graph Lipschitz domain $\Omega$, such that the modulus of the non-tangential limit of $G$ is $\nu$. 

As far as we know, the class $AE$ was introduced in Kenig's PhD thesis \cite{kenigSymposiaWeighted} and \cite{kenigWeighted}. He proved that, when the measures are in $A_{\infty}(\Lambda)$, this kind of analytic extensions exist, and are well-behaved in some sense that we will explore below. Later, Jerison and Kenig himself modified the definition in \cite{jerisonKenig}, adding a Smirnov-type property (Definition \ref{def:smirnov}).

To the best of our knowledge, this class has not been used until its recent appearance in \cite{CNO}, where Carro, Naibo and Ortiz-Caraballo use it to manipulate the norms of certain maximal operators that arise when solving the Neumann problem in graph Lipschitz domains in the plane with arc-length measure on the boundary. The motivation behind this note has been to understand better the class $AE$ in order to deal in \cite{Ainfty_BYC} with the Neumann problem with boundary measures more general than arc-length measure and extend some results in \cite{CNO}.

The structure of the article is the following. First, we reformulate
some ideas from \cite{HMW}. Then, we motivate the definition of the class $AE$ with a basic example. In view of its behaviour, we state and prove a general equivalence between an $H^1$-condition and a Smirnov-type condition in Theorem \ref{thm:generalizacionArgumentoH1}. Based on this, we re-define the class $AE$ and explain how, with the Smirnov-type property added in \cite{jerisonKenig}, the $H^1$-integrability property in the original definition of \cite{kenigWeighted} is no longer necessary (see Remarks \ref{remark:oneWeightUnnecessary} and \ref{remark:twoWeightUnnecessary} below). Finally, in Corollary \ref{corol:analogoLema4.4CNO} we generalize Lemma 4.4 of \cite{CNO} as an easy consequence of the theory developed in the previous sections. This generalization will be used in \cite{Ainfty_BYC} to deal with the $L^p$ and $H_{at}^1$-Neumann problem with $A_{\infty}$-measures on the boundary.

We would like to emphasize that certain ideas already appear in \cite{HMW}, \cite{kenigWeighted} and \cite{jerisonKenig}. However, we use them in a way that allows us to simplify the theory and obtain specific applications of it in a direct manner.

\section{Notation and preliminaries}\label{sec:section1.1_notation}

Let $\Lambda$ be a curve in the complex plane given parametrically by $\eta(x)=x+i\gamma(x)$ for $x \in \mathbb{R}$, where $\gamma$ is a real-valued Lipschitz function, and consider the graph Lipschitz domain above $\Lambda$,
\begin{equation}\label{eq:omega}
\Omega
=
\left\{
z_1+i z_2 \in \mathbb{C}: z_2>\gamma(z_1)
\right\}.    
\end{equation}
Note that $\partial \Omega=\Lambda$. Let $L$ denote the Lipschitz constant of $\gamma$.

We denote by $ds$ arc-length measure on $\Lambda$, defined as
\[
ds(E)
\coloneqq 
\int_{\eta^{-1}(E)}|\eta'(t)|\, dt
=
\int_{\eta^{-1}(E)}\sqrt{1+|\gamma'(t)|^2}\, dt,\quad E\subset\Lambda.
\]

\begin{definition}
Given $F$ a complex-valued function defined in $\Omega$ and $\alpha\in (0,\arctan{\frac{1}{L}})$, define the \textit{non-tangential maximal operator} $\mathcal{M}_\alpha$ of $F$ as
\begin{equation}\label{eq:definitionMalpha}
\mathcal{M}_\alpha(F)(\xi)=\sup _{z \in \Gamma_\alpha(\xi)}|F(z)|, \quad \xi \in \Lambda,    
\end{equation}
where
\[
\Gamma_\alpha(\xi)
=
\left\{
z_1+i z_2 \in \mathbb{C}: z_2>\operatorname{Im}(\xi)
\text { and }
|\operatorname{Re}(\xi)-z_1|<\tan (\alpha)|z_2-\operatorname{Im}(\xi)|
\right\}.
\]
\end{definition}

\begin{definition}
    Given a complex-valued function $F$ defined on $\Omega$, we say that $F(z)$ converges non-tangentially when $z\to \xi\in\partial\Omega$ if the following limit exists
    \[
    \lim_{z\triangleright \xi}F(z)
    \coloneqq 
    \lim_{\substack{z\to \xi\\ z\in\Gamma_{\alpha}(\xi)}}F(z),
    \]
    for some $\alpha\in (0,\arctan{\frac{1}{L}})$. In that case, we call this limit the non-tangential limit of $F$ when $z\to\xi$.
\end{definition}

\subsection{Conformal mapping}

Since $\Omega$ is simply connected, it is conformally equivalent to the upper half plane
\[
\R_{+}^{2}
\coloneqq 
\{
x+iy\in\C: x\in\R, y>0
\}.
\]
Let $z_0=i x_0$ with $x_0>\gamma(0)$. Let $\Phi\colon \R_{+}^{2} \rightarrow \Omega$ be the conformal mapping such that $\Phi(\infty)=\infty$, and $\Phi(i)=z_0$. Let $\Phi^{-1}: \Omega \rightarrow \R_{+}^{2}$ be its inverse. We have the following boundary behaviour.

\begin{theorem}[Thm. 1.1 in \cite{kenigWeighted}]\label{thm:thm1.1_Kenig}
$\Phi$ extends to $\overline{\R_{+}^{2}}$ as a homeomorphism onto $\overline{\Omega}$. Besides, $\Phi'$ has a non-tangential limit $dx$-a.e. on $\R$, and this limit is different from 0 $dx$-a.e. on $\R$. Moreover, $\Phi(x)$, with $x\in\R=\partial\R_{+}^{2}$, is absolutely continuous when restricted to any finite interval, and hence $\Phi'(x)$ exists $dx$-a.e. and is locally integrable. Moreover, this derivative coincides with the non-tangential limit of the derivative of the conformal map $\Phi$, i.e.,  $\Phi'(x)=\lim_{\R_{+}^{2}\ni z\triangleright x} \Phi'(z)$ $dx$-a.e.
\end{theorem}
We abuse notation and denote by $\Phi$ and $\Phi'$ the mappings in $\R_{+}^{2}$, the extensions to $\overline{\R_{+}^{2}}$ and the restrictions to $\R\equiv \partial\R_{+}^{2}$.

\subsection{Weights }

Muckenhoupt's weights (\cite{muckenhoupt}) can be defined on $\Lambda$ as is done for $\R$, just replacing Lebesgue measure $dx$ by arc-length measure $ds$, and intervals of $\R$ by the following notion of interval on $\Lambda$. See pages 133-134 of \cite{kenigWeighted} for more details.

\begin{definition}\label{def:intervalLambda}
    We define an \textit{interval} or arc in $\Lambda$ as the image of an interval $I\subset \R$ by the mapping $\eta$. That is, $J\subset\Lambda$ is an interval in $\Lambda$ if $\eta^{-1}(J)$ is an interval in $\R$. 
\end{definition}

Let $\nu$ be a measure over $\Lambda$ absolutely continuous with respect to arc-length $ds$. Define the following measure on $\R$ associated to $\nu$ via $\Phi$,
\[
\Phi(\nu)(E)
\coloneqq 
\nu(\Phi(E)),\quad E\subset \R.
\]
It is absolutely continuous with respect to $dx$ and its density function is 
$\frac{d\Phi(\nu)}{dx}(x)=\frac{d\nu}{ds}(\Phi(x)) |\Phi'(x)|$. We will abuse notation and write $\nu$ and $\Phi(\nu)$ instead of $\frac{d\nu}{ds}$ and $\frac{d\Phi(\nu)}{dx}$ respectively.

\begin{lemma}[Lemma 1.16 in \cite{kenigWeighted}]\label{lemma:phiNuAInfty}
With the above notation, 
\begin{equation}\label{eq:AinftyEquivalence}
\nu\in A_{\infty}(\Lambda)\quad
\text{ if and only if }\quad
\Phi(\nu)\in A_{\infty}(\R).    
\end{equation}
\end{lemma}

\subsection{Hardy spaces over graph Lipschitz domains}

Hardy spaces are useful to manipulate the $L^p$-norms of the non-tangential maximal operators $\mathcal{M}_{\alpha}$ defined above.

\begin{definition}[Definition 2.7 in \cite{kenigWeighted}]
Let $\nu$ be a measure on $\Lambda$ and let $\alpha\in (0,\arctan{\frac{1}{L}})$. Define the Hardy space 
\[
H^p(\Omega, \nu)
\coloneqq 
\left\{
h: \Omega \rightarrow \mathbb{C}: h \text { is analytic in } \Omega \text { and }\left\|\mathcal{M}_\alpha(h)\right\|_{L^p(\Lambda, \nu)}<\infty
\right\}
\]
and set $\|h\|_{H^p(\Omega, \nu)}\coloneqq \left\|\mathcal{M}_\alpha(h)\right\|_{L^p(\Lambda, \nu)}$.
\end{definition}

The definition of $H^p(\Omega, \nu)$ is independent of $\alpha$. Different values of $\alpha$ give rise to equivalent norms. There is a characterization of this space in terms of $L^p$-norms over curves.

\begin{theorem}[Theorem 2.13 in \cite{kenigWeighted}]
    Let $\nu\in A_{\infty}(\Lambda)$ and $p\in (0,\infty)$. Then, $F\in H^p(\Omega,\nu)$ if and only if 
    \[
    \sup_{h>0}\int_{\Lambda}|F(\xi+ih)|^p\, d\nu(\xi)
    <
    \infty.
    \]
\end{theorem}

In particular, when $\Omega=\R_{+}^{2}$, for $w\in A_{\infty}(\R)$ and $p\in (0,\infty)$ we have 
\[
F\in H^p(\R_{+}^{2},w)
\iff 
\sup_{y>0}\|F(\cdot+iy)\|_{L^p(\R,w)}
<
\infty.
\]

\subsection{Poisson integrals and \texorpdfstring{$A_{\infty}$}{}-weights}

\begin{definition}
    We define the Poisson kernel at $(x,y)\in\R_{+}^{2}$ as
    \[
    P_{y}(x)
    \coloneqq 
    \frac{1}{\pi}\frac{y}{x^2+y^2}.
    \]
    Given a function $f$ with $\int_{\R}\frac{|f(x)|}{1+|x|^2}\, dx<\infty$, we define the Poisson integral of $f$ at $(x,y)\in\R_{+}^{2}$ as
    \[
    (P_y\ast f)(x)
    \coloneqq 
    \frac{1}{\pi}\int_{\R}\frac{y}{(x-t)^2+y^2}f(t)\, dt.
    \]
\end{definition}

One result that will turn useful for us is the following (see page 62 of \cite{garnett}).

\begin{theorem}\label{thm:log_garnettP62}
    Given $p\in (0,\infty]$, if $f\in H^p(\R_{+}^{2},dx)$ and $f(z)\neq 0$ for every $z\in\R_{+}^{2}$, then
    \[
    \log{|f(z)|}
    \leq 
    P_{y}\ast \log{|f|}(x), \quad  \forall z=x+iy\in\R_{+}^{2},
    \]
    where the function $f$ on the right-hand side denotes the non-tangential limits of $f(z)$.
\end{theorem}

One class of functions $f$ that satisfy $\int_{\R}\frac{|f(x)|}{1+|x|^2}\, dx<\infty$ is $BMO$ (see pages 141-142 of \cite{FeffermanSteinHp}). Our source of $BMO$ functions will be logarithms of $A_{\infty}$-weights. 

\begin{lemma}[Corollary IV.2.19 in \cite{GCRdF}]\label{lemma:AinftyBMO}
    If $w\in A_{\infty}$, then $\log{w}\in BMO$.
    As a consequence, given $w_1,w_2\in A_{\infty}$ and $\alpha_1,\alpha_2\in\R$, then $\log(w_1^{\alpha_1}w_2^{\alpha_2})\in BMO$. 
\end{lemma}

Finally, we write $A\lesssim B$ when there exists a constant $C>0$ such that $A\leq C\cdot B$. If we want to make explicit the dependence of $C$ on some parameter $\alpha$, we may write $A\lesssim_{\alpha} B$. We write $A\simeq B$ to denote $A\lesssim B$ and $B\lesssim A$.

\section{Hunt-Muckenhoupt-Wheeden's argument}

The article \cite{HMW} was a breakthrough in the theory of weights. It characterized the weights $w$ for which the Hilbert transform is bounded on $L^p(w)$. In order to do so, several new properties and techniques regarding $A_p(\R)$-weights were introduced. These have been used and extended in \cite{kenigWeighted} and \cite{CNO} in order to deal with boundary value problems in graph Lipschitz domains.

In particular, we are interested in an argument handled in pages 248-249 of \cite{HMW}. Kenig refers to it repeatedly in \cite{kenigWeighted} to prove estimates of the form  $\frac{F(z)}{(i+z)^m}\in H^1(\R_{+}^{2},dx)$ for some $m\geq 0$ and some analytic function $F$. Since it is essential for the class $AE$, we have tried to express the argument in a closed statement as general as possible, so that we can refer to it when needed.

To begin with, we state some basic results. From now on, we denote $A_p(\R)$ by $A_p$, for $p\in [1,\infty]$.

\begin{lemma}\label{lemma:ApIntegrability}
    If $w\in A_p$ with $p\in (1,\infty)$, then 
    \[
    \frac{w(x)}{(1+|x|)^p}\in L^1(\R,dx).
    \]
\end{lemma}

\begin{proof}
    Apply the first item of Lemma 2.3 (a) in \cite{CNO}
    with $t=1$, $x_0=0$, to obtain
    \[
    \int_{\R}\frac{w(x)}{(1+|x|)^p}\, dx
    \lesssim 
    \frac{1}{1^p}\int_{|x|<1}w(x)\, dx
    <
    \infty.\qedhere
    \]
\end{proof}

\begin{corollary}\label{corol:poissonIntegralA2Weight}
    The Poisson integral of an $A_2$ weight is well-defined in $\R_{+}^{2}$.
\end{corollary}

\begin{lemma}\label{lemma:lemmaAuxiliarHMW1}
    If $W\in A_2$, then
    \[
    (P_y\ast W)(x)
    \lesssim 
    \frac{1}{y}\int_{|x-t|<y}W(t)\, dt.
    \]
\end{lemma}

\begin{proof}
    We have 
    \[
    \int_{\R}\frac{y}{(x-t)^2+y^2}W(t)\, dt
    \lesssim 
    \frac{1}{y}\int_{|x-t|<y}W(t)\, dt
    +
    y\int_{|x-t|>y}\frac{W(t)}{|x-t|^2}\, dt  
    \lesssim 
    \frac{1}{y}\int_{|x-t|<y}W(t)\, dt,
    \]
    having applied in the last inequality Lemma 1 in \cite{HMW} with $p=2$ and $I=[x-y,x+y]$, so that $|I|=2y$ and $a_I=x$.
\end{proof}

The attempt to formulate the argument given in pages 248-249 of \cite{HMW} is the following. It will play a fundamental role later.

\begin{theorem}\label{thm:lemaHMWmejor}
    Let $W$ be a holomorphic function in $\R_{+}^{2}$ such that its modulus can be majorized by the Poisson integral of an $A_2$ weight; i.e., there exists $\bm{W}\in A_2$ such that
        \begin{equation}\label{eq:lemaHMWsizeCondition}
        |W(x+iy)|
        \lesssim 
        C(P_y\ast \bm{W})(x),\quad (x,y)\in\R_{+}^{2}. 
        \end{equation}
    Let $F$ be a bounded holomorphic function in $\R_{+}^{2}$ such that $|F(z)|=O(|z|^{-1})$ when $|z|\to\infty$.

    Then
    \[
    F^2\cdot W
    \in 
    H^1(\R_{+}^{2},dx).
    \]
\end{theorem}

\begin{proof}
    The assumption \eqref{eq:lemaHMWsizeCondition} together with Lemma \ref{lemma:lemmaAuxiliarHMW1} give us
    \begin{equation}\label{eq:ineqProofLemmaHMW2}
    |W(x+iy)|
    \lesssim 
    \frac{1}{y}\int_{|x-t|<y}\bm{W}(t)\, dt.    
    \end{equation}
    Fix $y>0$. Using that $F$ is bounded with $F(z)=O(|z|^{-1})$, inequality \eqref{eq:ineqProofLemmaHMW2} and Tonelli's theorem, we obtain
    \[
    \int_{\R}|F(x+iy)|^2|W(x+iy)|\, dx
    \lesssim
    \frac{1}{y}\int_{t\in\R}\bm{W}(t)\left(\int_{x:|x-t|<y}\frac{1}{1+x^2+y^2}\, dx\right)\, dt.   
    \]
    If $|x-t|<y$, then
    \[
    t^2
    =
    (x+(t-x))^2
    \leq 
    2(x^2+(x-t)^2)
    \leq 
    2(x^2+y^2),
    \]
    so
    \[
    \int_{x:|x-t|<y}\frac{1}{1+x^2+y^2}\, dx
    \lesssim 
    \int_{x:|x-t|<y}\frac{1}{1+t^2}\, dx
    \simeq
    \frac{2y}{1+t^2}.
    \]
    Therefore, the above expression can be bounded by $\int_{\R}\frac{\bm{W}(t)}{1+t^2}\, dt$, which is finite by Lemma \ref{lemma:ApIntegrability}.  So $F^2\cdot W$ is uniformly integrable on horizontal lines, i.e., $F^2\cdot W\in H^1(\R_{+}^{2},dx)$.\qedhere    
\end{proof}

We will apply Theorem \ref{thm:lemaHMWmejor} several times in the following particular case, which employs the notion of Smirnov-type condition.

    \begin{definition}[Smirnov-type condition]\label{def:smirnov}
    Let $F\colon \R_{+}^{2}\to \C$ be a holomorphic function that does not vanish anywhere in $\R_{+}^{2}$. We say that $F$ is of \textit{Smirnov-type} in $\R_{+}^{2}$, or satisfies a \textit{Smirnov-type condition}, if 
    $\log{|F(z)|}$ has non-tangential limits $dx$-a.e., which we denote by $\log{|\Tilde{F}|}(x)$, and 
    \[
    \log{|F(z)|}
    =
    (P_y\ast \log{|\widetilde{F}|})(x),
    \quad\forall z=x+iy\in\R_{+}^{2}.
    \]
    \end{definition}

Notice that if $|F|$ has non-tangential limits $|\Tilde{F}|\in A_{\infty}$, then $\log{|F|}$ has non-tangential limits $\log{|\Tilde{F}|}\in BMO$, and therefore the Poisson integral on the right-hand side is well-defined.

The next result is the first one pointing towards a relation between Smirnov-type and $H^1$-integrability when the non-tangential limit is in $A_{\infty}$.

\begin{corollary}\label{corol:lemaHMWCasoParticular}
    Let $W$ be a holomorphic function in $\R_{+}^{2}$ that does not vanish. Assume that $|W|$ has non-tangential limits $\bm{W}(x)$ a.e. $x\in\R$ with $\bm{W}\in A_2$, and that $W$ satisfies a Smirnov-type condition.
    
    Let $F$ be a bounded holomorphic function in $\R_{+}^{2}$ such that $|F(z)|=O(|z|^{-1})$ when $z\to\infty$.
    
    Then
    \[
    F^2\cdot W
    \in 
    H^1(\R_{+}^{2},dx).
    \]
\end{corollary}

\begin{proof}
    Just observe that the Smirnov-type property implies the bound \eqref{eq:lemaHMWsizeCondition}. Indeed, if $\log{|W(z)|}=(P_y\ast \log{|\bm{W}|})(x)$, then taking exponentials and applying Jensen's inequality we obtain
    \[
    |W(z)|
    =
    e^{(P_y\ast \log{|\bm{W}|})(x)}
    \leq 
    (P_y\ast e^{\log{|\bm{W}|}})(x)
    =
    (P_y\ast |\bm{W}|)(x),
    \]
    with $|\bm{W}|\in A_2$ by hypothesis. \qedhere
\end{proof}

\section{The class AE of analytic extensions}

In this section, we investigate the class $AE$. But first, for clarity purposes, we recall some known results for the derivative $\Phi'$ of the conformal map $\Phi$. These results and their proofs will serve us as a model for the more general theory.

\subsection{The basic example \texorpdfstring{$\Phi'$}{}}

The basic example that we want to mimic is $\Phi'(z)$, $z\in \R_{+}^{2}$, seen as an analytic extension of the measure $|\Phi'(x)|\, dx$ on $\R=\partial\R_{+}^{2}$.

We recall the following fact, and give the proof of the third item, which is not explicit in \cite{kenigWeighted} and will be meaningful for us. Since it makes use of the proof of the second item, we give it as well.

\begin{theorem}[Thm. 1.10 in \cite{kenigWeighted}]\label{thm:thm1.10_Kenig}
With the above definitions and notation:
\begin{enumerate}
    \item $\left|\arg \Phi^{\prime}(z)\right| \leq \arctan{L}<\frac{\pi}{2}$ for all $z\in\R_{+}^{2}$.
    \item The locally integrable function $|\Phi'(x)|$ on $\R$ is an $A_2$-weight: $|\Phi'| \in A_2$.
    \item For any $\varepsilon>0$, 
    \begin{equation}\label{eq:integrabilityConditionH1}
    \frac{\Phi'(z)}{(i+\varepsilon z)^2}\hspace{0.2cm}\text{ and }
    \hspace{0.2cm} 
    \frac{[\Phi'(z)]^{-1}}{(i+\varepsilon z)^2}\in H^1(\R_{+}^{2}, dx).    
    \end{equation}
\end{enumerate}
\end{theorem}

\begin{proof}

\textbf{Proof of the second item.} The function $\Phi'(z)$ is analytic in $\R_{+}^{2}$ and does not vanish. Besides, by the previous item, there exists $\varepsilon>0$ such that $|\arg{\Phi'(z)}|\leq\frac{\pi}{2}-\varepsilon$. Therefore, 
\[
\log{\Phi'(z)}
=
\log{|\Phi'(z)|}+i\arg{\Phi'(z)},\quad z\in \R_{+}^{2},
\]
is also an analytic function in $\R_{+}^{2}$. 
Taking non-tangential limits,
\[
\lim_{z\triangleright x}\log{\Phi'(z)}
=
\log{|\Phi'(x)|}
+
i\arg{\Phi'(x)}
\eqqcolon 
g_1(x)+ig_2(x),\quad dx\text{-a.e. in }\R,
\]
with $\|g_2\|_{\infty}\leq \frac{\pi}{2}-\varepsilon$. The fact that $|\arg{\Phi'(z)}|\leq\frac{\pi}{2}-\varepsilon$ also implies, by the case $p=\infty$ of Theorem II.2.5 
of \cite{steinWeiss}, that
\[
\arg{\Phi'(z)}=P_y\ast g_2(x),\quad z=x+iy\in\R_{+}^{2}.
\]
Therefore, $\log{\Phi'(z)}$ and $-P_y\ast Kg_2(x)+iP_y\ast g_2(x)$ are two holomorphic\footnote{Here, $K$ denotes the Hilbert transform of $L^{\infty}$-functions; see for example page 105 of \cite{garnett} or page 135 of \cite{kenigWeighted}.} functions with the same imaginary parts. Therefore, by the Cauchy-Riemann equations, the real parts have to be equal up to an additive constant: there exists $C>0$ such that
\[
\log{|\Phi'(z)|}
=
-P_y\ast Kg_2(x)+C, \quad z=x+iy\in\R_{+}^{2}.
\]
Notice that, since $\int_{\R}P_y(t)\, dt=1$ for all $y>0$, then $(P_y\ast C)(x)=C$ for any $(x,y)\in \R_{+}^{2}$. Hence,
\[
\log{\Phi'(z)}
=
P_y\ast (-Kg_2+C)(x)
+
i(P_y\ast g_2)(x),\quad z=x+iy\in\R_{+}^{2}.
\]
Taking exponentials on both sides, we obtain
\[
\Phi'(z)
=
e^{P_y\ast (-Kg_2+C)(x)}
e^{i (P_y\ast g_2)(x)},
\quad (x,y)\in \R_{+}^{2}.
\]
Its modulus is
\begin{equation}\label{eq:moduloPhiPrima}
|\Phi'(z)|
=
e^{P_y\ast (-Kg_2+C)(x)},\quad z=x+iy\in\R_{+}^{2}.    
\end{equation}
Taking non-tangential limits, we get
\begin{equation}\label{eq:boundaryValuesModuloPhiPrima}
|\Phi'(x)|
=
e^{-Kg_2(x)+C},\quad x\in\R,    
\end{equation}
with $\|g_2\|_{\infty}\leq \frac{\pi}{2}-\varepsilon$. Denoting $f_1\coloneqq C$, $f_2\coloneqq -g_2$, we obtain that 
\[
|\Phi'|=e^{f_1+Kf_2},
\quad f_1,f_2\in L^{\infty},
\|f_2\|_{\infty}<\frac{\pi}{2}.
\]
So, by Helson-Szegö (Theorem I.8.14 in \cite{GCRdF}), $|\Phi'|\in A_2$.

\vspace{0.2cm}

\textbf{Proof of the third item.} Notice that taking logarithms in \eqref{eq:moduloPhiPrima} and \eqref{eq:boundaryValuesModuloPhiPrima}, one obtains that 
\begin{equation}\label{eq:smirnovPhiPrima}
\log{|\Phi'(z)|}
=
P_y\ast \log{|\Phi'|}(x),
\end{equation}
i.e., $\Phi'$ satisfies the Smirnov type-condition in $\R_{+}^{2}$, so Corollary \ref{corol:lemaHMWCasoParticular} with $W=\Phi'(z)$ and $F(z)=\frac{1}{i+\varepsilon z}$ gives us the first inclusion, and with $W=\Phi'(z)^{-1}$ and $F(z)=\frac{1}{i+\varepsilon z}$ gives us the second one.\qedhere  
\end{proof}

\subsection{Relation between Smirnov-type and \texorpdfstring{$H^1$}{}-integrability}

The key to prove the third item of Theorem \ref{thm:thm1.10_Kenig} has been the Smirnov condition \eqref{eq:smirnovPhiPrima} and the fact that $|\Phi'|$ and $|\Phi'|^{-1}$ are $A_2$-weights. This leads us to think that the above argument can be generalized to $A_{\infty}$-measures, because if $w\in A_{\infty}$, there exists a negative power of $w$ in  $A_{\infty}$, and the situation is then analogous. In fact, we obtain not only that Smirnov-type implies $H^1$-integrability, but also that the converse is true. First, we state the following simple but very useful fact.

    \begin{lemma}\label{lemma:smirnovProp1}
    Let $G_1$ and $G_2$ be non-vanishing holomorphic functions in $\R_{+}^{2}$, of Smirnov-type, and such that $|G_1|$ and $|G_2|$ have non-tangential limits in $A_{\infty}$. Let $\alpha_1,\alpha_2\in\R$. Then $G_1^{\alpha_1}G_2^{\alpha_2}$ is of Smirnov-type.
    \end{lemma}

Our main result is the following.

\begin{theorem}\label{thm:generalizacionArgumentoH1}
    Let $F\colon \R_{+}^{2}\to \C$ be a non-vanishing holomorphic function such that $|F|$ has non-tangential limit $|\Tilde{F}|\in A_{\infty}$.
    Then, $F$ satisfies a Smirnov-type condition if and only if there exist $m,s\geq 0$ and $q>0$ such that 
    \begin{equation}\label{eq:H1integrabilityViaSmirnov}
    \frac{F(z)}{(i+z)^m}\in H^1(\R_{+}^{2},dx),
    \quad
    \frac{F(z)^{-q}}{(i+z)^s}\in H^1(\R_{+}^{2},dx).
    \end{equation}
\end{theorem}

\begin{proof}

    Assume first that $F$ satisfies a Smirnov-type condition. We prove the first inclusion of \eqref{eq:H1integrabilityViaSmirnov}. The ideas and notation follow closely those of Lemma 2.12 in \cite{kenigWeighted}.
    
    By hypothesis, $|\Tilde{F}|\in A_{\infty}$, so there exists $r>1$ such that $|\Tilde{F}|\in A_r$. We distinguish two cases.

    \textit{Case $1$:} $r<2$. In this case $|\Tilde{F}|\in A_r\subset A_2$ and $F$ is of Smirnov-type. Hence, by Corollary \ref{corol:lemaHMWCasoParticular}, $\frac{F(z)}{(i+z)^2} \in H^1(\R_{+}^{2},dx)$.

    \textit{Case $2$:} $r>2$. The idea is to reduce to the previous case. Let $a\coloneqq \frac{1}{r-1}$. Recall that $w\in A_r$ if and only if $w^{-\frac{1}{r-1}}\in A_{r'}$. Therefore, since $|\Tilde{F}|\in A_r$, we have $|\Tilde{F}|^{-a}\in A_{r'}\subset A_2$, and hence 
    \[
    |\Tilde{F}|^{a}\in A_2.
    \]
    Let 
    \[
    H_a(z)\coloneqq F(z)^{a}.
    \] 
    By Lemma \ref{lemma:smirnovProp1}, this function satisfies a Smirnov-type condition. 
    Besides, $|H_a|$ converges non-tangentially to $|\Tilde{F}|^{a}\in A_2$. Therefore, we can apply Corollary \ref{corol:lemaHMWCasoParticular} and obtain 
    \[
    \frac{H_a(z)}{(i+z)^2}
    \in 
    H^1(\R_{+}^{2},dx).
    \]
    This is equivalent to
    \[
    \frac{F(z)}{(i+z)^{\frac{2}{a}}}\in H^{a}(\R_{+}^{2},dx).
    \]
    Its modulus has non-tangential limit $\frac{|\Tilde{F}|(x)}{|i+x|^{\frac{2}{a}}}$, which is in $L^1(\R)$. Indeed, since $|\Tilde{F}|\in A_r$, by Lemma \ref{lemma:ApIntegrability} we have $\frac{|\Tilde{F}|(x)}{|i+x|^r}\in L^1(\R)$. Therefore, $\frac{|\Tilde{F}|(x)}{|i+x|^{\frac{2}{a}}}\in L^1(\R)$, because $\frac{2}{a}=2(r-1)=r+(r-2)>r$. 
    
    Since $\frac{F(z)}{(i+z)^{\frac{2}{a}}}\in H^{a}(\R_{+}^2,dx)$ with non-tangential limit in $L^1(\R)$, then, by Corollary II.4.3 
    of \cite{garnett},
    \[
    \frac{F(z)}{(i+z)^{\frac{2}{a}}}\in H^1(\R_{+}^{2},dx),
    \]
    obtaining statement with $m=\frac{2}{a}$.
    
    To see the second inclusion in \eqref{eq:H1integrabilityViaSmirnov}, notice that $|\widetilde{F}|\in A_r$ implies $|\widetilde{F}|^{-\frac{1}{r-1}}\in A_{r'}$. Taking $q=-\frac{1}{r-1}$ and following the same argument as for $F$, but now for $F^q$, we obtain the desired conclusion.

    \vspace{0.3cm}
    
    Now, we prove the converse. First, notice that
    \[
    \frac{1}{|i+z|^m}\simeq_{\varepsilon}\frac{1}{|i+\varepsilon z|^m},\quad \forall \varepsilon>0.
    \]
    Hence,
    \[
    \frac{F(z)}{(i+\varepsilon z)^m}\in H^1(\R_{+}^{2},dx),\quad \forall \varepsilon>0.
    \]
    So, by Theorem \ref{thm:log_garnettP62}, 
    \[
    \log{\frac{|F(z)|}{|i+\varepsilon z|^m}}
    \leq 
    \left(P_y\ast \log{\frac{|\Tilde{F}(\cdot)|}{|i+\varepsilon (\cdot)|^m}}\right)(x),\quad \forall \varepsilon>0.
    \]
    The left-hand side tends to $\log{|F(z)|}$ when $\varepsilon\to 0^{+}$, by continuity in $\varepsilon$. The right-hand side tends to $P_{y}\ast \log{|\Tilde{F}|}(x)$ by the Dominated Convergence Theorem, which can be applied because the integrand in the convolution can be bounded by $P_{y}(x-t)\log{|\Tilde{F}(t)|}\in L^1$, since $\log{|\Tilde{F}(t)|}\in BMO$. 
    Therefore,
    \[
    \log{|F(z)|}
    \leq 
    P_{y}\ast \log{|\Tilde{F}|}(x).
    \]

    Reasoning in the same way with $\frac{F(z)^{-q}}{(i+z)^m}$, we obtain
    \[
    \log{|F(z)|^{-q}}
    \leq 
    P_{y}\ast \log{|\Tilde{F}|^{-q}}(x),
    \]
    which gives the other inequality $\log{|F(z)|}
    \geq
    P_{y}\ast \log{|\Tilde{F}|}(x)$.
\end{proof}

\subsection{The class AE of analytic extensions in \texorpdfstring{$\R_{+}^{2}$}{}}

The properties about $\Phi'$ proved in the previous section motivate the following definition.

\begin{definition}[Analytic extensions on $\R_{+}^{2}$]\label{def:definitionAEsemiplano}
    Given $\mu\in A_{\infty}(\R)$, 
    we say that a function $F$ is an analytic extension of $\mu$ to $\R_{+}^{2}$, and we write $H\in AE(\R_{+}^{2},\mu)$ or simply $H\in AE(\mu)$, if:
    \begin{itemize}
        \item $H$ is an analytic function in $\R_{+}^{2}$ that does not vanish.
        \item $H$ has non-tangential limits $\Tilde{H}$ $dx$-a.e. such that $|\Tilde{H}(x)|\simeq \frac{d\mu}{dx}(x)$ $dx$-a.e.
        \item $H$ satisfies a Smirnov-type condition on $\R_{+}^{2}$. 
    \end{itemize}
\end{definition}

\begin{remark}\label{remark:oneWeightUnnecessary}
    As mentioned in the Introduction, the original definition (Definition 2.10 in \cite{kenigWeighted}) included an $H^1$-integrability condition: it asked for the existence of $m\geq 0$ such that
    \begin{equation}\label{eq:H1conditionDefinition}
    \frac{H(z)}{(i+z)^m}\in H^1(\R_{+}^{2},dx).    
    \end{equation} 
    Afterwards, in page 232 of \cite{jerisonKenig}, the Smirnov-type condition was included. By Theorem \ref{thm:generalizacionArgumentoH1}, this condition makes \eqref{eq:H1conditionDefinition} unnecessary. This simplifies the theory, because the $H^1$-integrability was the difficult part to check in particular cases, as can be seen in \cite{CNO}.
\end{remark}

As a consequence of the results in the previous section, we have:

\begin{corollary}\label{corol:phiPrimeInAE}
    With the above notation, $\Phi'\in AE(\R_{+}^{2},|\Phi'(x)|dx)$.    
\end{corollary}

In fact, since the Smirnov-type property behaves well with respect to powers, we have more:

\begin{corollary}\label{lemma:potenciaPhiPrimaAE}
    $(\Phi')^{\alpha}\in AE(\R_{+}^{2},|\Phi'|^{\alpha})$ for every $\alpha\in \R$ for which $|\Phi'|^{\alpha}\in A_{\infty}(\R)$.
\end{corollary}

\begin{proof}
    Recall that $(\Phi')^{\alpha}(z)\coloneqq e^{\alpha\log{\Phi'(z)}}$, where the logarithm can be defined using the principal branch, as observed in the proof of Theorem \ref{thm:thm1.10_Kenig}.     
    This function is holomorphic in $\R_{+}^{2}$, does not vanish, and has non-tangential limit  
    $e^{\alpha\log{\Phi'(x)}}$, with modulus $e^{\alpha\log{|\Phi'(x)|}}=|\Phi'(x)|^{\alpha}$. And by Corollary \ref{corol:phiPrimeInAE} and Lemma \ref{lemma:smirnovProp1}, $(\Phi')^{\alpha}$ satisfies Smirnov.
\end{proof}

\begin{proposition}[Lemma 2.12 of \cite{kenigWeighted}]\label{thm:thm2.12_Kenig}
        Let $w\in A_{\infty}$. Then\footnote{Formula \eqref{eq:elementInAE} is morally correct, but the convolution $Q_t\ast \log{w}$ is not well-defined, and therefore we must clarify what it means. Decompose $\log{w}\in BMO$ as $\log{w}=f_1+Kf_2$ with $f_1,f_2\in L^{\infty}$, and then define 
        \[
        H(z)
        =
        e^{P_{y}\ast (f_1+Kf_2)(x)+iP_{y}\ast(Kf_1-f_2)(x)},\quad z=x+iy\in\R_{+}^{2}.
        \]
        The term $P_{y}\ast(Kf_1-f_2)$ is the formalisation of the heuristic notation $Q_y\ast \log{w}$ or, equivalently, $P_y\ast K(\log{w})$. The exponent is then holomorphic because it is the sum of
        $P_y\ast (f_1+i Kf_1)$ and $(-i)P_y\ast (f_2+iKf_2)$.
        } 
        \begin{equation}\label{eq:elementInAE}
        H(z)
        \coloneqq 
        e^{(P_y+iQ_y)\ast \log{w}(x)},
        \quad z=x+iy\in\R_{+}^{2},    
        \end{equation}
        is a holomorphic function in $\R_{+}^{2}$ that does not vanish, has non-tangential limit with modulus $w$ and is of Smirnov-type in $\R_{+}^{2}$. That is, $H\in AE(\R_{+}^{2},w)$.
    \end{proposition}

\subsection{The class AE of analytic extensions in \texorpdfstring{$\Omega$}{}}

We can also consider analytic extensions in $\Omega$ of measures defined on $\Lambda$. We begin again with a simple case. 

Imagine we want to have an analytic extension of arc-length measure $ds$. Its density function is $\frac{ds}{ds}(\xi)\equiv 1$, so the obvious analytic extension of this measure is $G(z)\equiv 1$ in $\Omega$.

Also notice that $\Phi(ds)=|\Phi'(x)|dx$, and the analytic extension of $ds$, which is $1$, can be obtained from the analytic extension of $|\Phi'(x)|$, which is $\Phi'(z)$, by the relation
\[
\Phi'(z)
=
(1\circ\Phi)(z)\cdot \Phi'(z),\quad z\in\R_{+}^{2}.
\]
In view of this example and formula $\Phi(\nu)=(\nu\circ\Phi)|\Phi'|$, it makes sense to define analytic extensions in $\Omega$ in the following way.

\begin{definition}
Let $\Omega$ be a graph Lipschitz domain and let $\Phi\colon \R_{+}^{2}\to\Omega$ be the usual conformal mapping. Let $\nu$ be a measure on $\Lambda$ which is absolutely continuous with respect to arc-length measure. We say that $G\colon \Omega\to \C$ is an analytic extension of $\nu$ to $\Omega$, and write $G\in AE(\Omega,\nu)$, if
\begin{equation}\label{eq:equivalentConditionAE}
(G\circ\Phi)\cdot\Phi'\in AE(\R_{+}^{2},\Phi(\nu)).    
\end{equation}
\end{definition}

Notice that $AE(\Omega,\nu)\neq\varnothing$ if $\nu\in A_{\infty}(\Lambda)$. Indeed, if $\nu\in A_{\infty}(\Lambda)$ then $\Phi(\nu)\in A_{\infty}$. Therefore, by Theorem \ref{thm:thm2.12_Kenig}, $AE(\R_{+}^{2},\Phi(\nu))\neq\varnothing$. Take $H\in AE(\R_{+}^{2},\Phi(\nu))$, then $(H\circ\Phi^{-1})\cdot (\Phi^{-1})'\in AE(\Omega,\nu)$.

\subsection{Analytic extensions relating two measures}

It is possible to define a class of analytic extensions that relates two measures $\nu$ and $\mu$ in $\R_{+}^{2}$, by extending the Radon-Nikodym derivative
$\frac{d\nu}{d\mu}$.

\begin{definition}\label{def:AENuMu}
    Let $\nu,\mu\in A_{\infty}$. The space $AE(\R_{+}^{2},\nu; \R_{+}^{2},\mu)$, abbreviated $AE(\nu,\mu)$, is the class of functions $h\colon \R_{+}^{2}\to\C$ such that
    \begin{itemize}
        \item $h$ is analytic on $\R_{+}^{2}$ and different from $0$ at every point.
        \item $h$ has non-tangential limit $\widetilde{h}(x)$ $dx$-a.e. in $\R$, and this limit satisfies $|\widetilde{h}(x)|\simeq \frac{d\nu}{d\mu}(x)$ for $dx$-a.e.
        \item The function $h$ satisfies a Smirnov-type condition on $\R_{+}^{2}$.
    \end{itemize}  
\end{definition}

\begin{remark}\label{remark:twoWeightUnnecessary}
    Again, the original definition (Definition 2.11 of \cite{kenigWeighted}) included an $H^1$-integrability condition that is no longer necessary, because it is implied by the Smirnov-type condition.
\end{remark}

\begin{theorem}[Thm. 2.12 of \cite{kenigWeighted}]\label{thm:thm2.12_Kenig_Bis}
    Let $\nu,\mu\in A_{\infty}$. Then $AE(\nu,\mu)\neq \varnothing$.
\end{theorem}

\begin{proof}
By Theorem \ref{thm:thm2.12_Kenig}, there exist $R\in AE(\R_{+}^{2},\nu)$ and $H\in AE(\R_{+}^{2},\mu)$, so $h\coloneqq \frac{R}{H}\in AE(\nu,\mu)$.\qedhere
\end{proof}

As seen in this proof, $AE(\nu,\mu)$ can be seen as quotients of functions in $AE(\nu)$ and $AE(\mu)$. Indeed, $h\in AE(\nu,\mu)$ if and only if $h\cdot \bm{k}\in AE(\nu)$ for any $\bm{k}\in AE(\mu)$.

With the same proof of Lemma \ref{lemma:potenciaPhiPrimaAE}, we obtain:

\begin{lemma}\label{lemma:PhiPrimaAlphaAE}
    If $\nu,\mu\in A_{\infty}(\R)$ and $\frac{d\nu}{d\mu}=|\Phi'|^{\alpha}$ for some $\alpha\in\R$, then $(\Phi')^{\alpha}\in AE(\nu,\mu)$.
\end{lemma}

\subsection{Applications}

The aim of this section is to apply the previous theory to generalize Lemma 4.4 of \cite{CNO} in a simple manner. In order to do that, we need to recall a result from\footnote{We notice that, since our definitions of $AE(\Omega,\nu)$ and $AE(\nu,\mu)$ are at least as restrictive as the ones given in \cite{kenigWeighted}, we can apply all the theorems proved in there for these classes.} \cite{kenigWeighted}.

\begin{theorem}[Corollary 2.18 of \cite{kenigWeighted}]\label{thm:thm2.18_Kenig}
    Let $p\in (0,\infty)$ and $\nu,\mu\in A_{\infty}$. If $h\in AE(\nu, \mu)$, then 
    \[
    F \in H^p(\R_{+}^{2},\nu) 
    \iff 
    F\cdot h^{\frac{1}{p}} \in H^p(\R_{+}^{2},\mu),
    \]
    with equivalent norms.
\end{theorem}

This result makes rigorous the heuristic computation done in \eqref{eq:desiredChangeMeasureHpNorm}, expressing $\nu=\mu\cdot \frac{d\nu}{d\mu}$ and viewing $h$ as the analytic extension of $\frac{d\nu}{d\mu}$.

As a consequence of Lemma \ref{lemma:PhiPrimaAlphaAE} and Theorem \ref{thm:thm2.18_Kenig}, we get the following generalization of Lemma 4.4 of \cite{CNO}, which will be useful for the study of the $L^{p}$-solvability and the $H_{at}^1$-solvability of the Neumann problem in \cite{Ainfty_BYC}. The generalization is two-fold. First, their measure is the particular case $\nu=ds$. Second, we ask for $|\Phi'|^{-p}\Phi(\nu)$ in $A_{\infty}$ instead of $A_p$.

\begin{corollary}\label{corol:analogoLema4.4CNO}
   
    Let $\nu\in A_{\infty}(\Lambda)$ and $p\in [1,\infty)$. Suppose that $|\Phi'|^{-p}\Phi(\nu)\in A_\infty$. Then
    \[
    F\in H^p(\R_{+}^{2},|\Phi'|^{-p}\Phi(\nu))
    \iff 
    F\cdot \frac{1}{\Phi'}\in H^p(\R_{+}^{2},\Phi(\nu)),
    \]
    with equivalent norms.
\end{corollary}

\begin{proof}
It follows from Theorem \ref{thm:thm2.18_Kenig} with $F\equiv F$, $h\equiv (\Phi')^{-p}$, $\bm{d\nu}\equiv |\Phi'|^{-p}\Phi(\nu)$ and $\bm{d\mu}\equiv \Phi(\nu)$. It can be applied because $\bm{\nu},\bm{\mu}\in A_{\infty}$ by hypothesis, and $h\in \operatorname{AE}(\bm{\nu},\bm{\mu})$ by Lemma \ref{lemma:PhiPrimaAlphaAE}. \qedhere
\end{proof}

In fact, the proof of the previous result would work for any weights $\nu,\mu\in A_{\infty}$ with $\frac{d\nu}{d\mu}\simeq |\Phi'|^{-p}$.

To end this section, we want to remark that $AE(\nu,\mu)$ is defined in greater generality in Definition 2.11 of \cite{kenigWeighted}, considering two measures defined on two different graph Lipschitz domains. However, we have restricted to $\R_{+}^{2}$ because, for applications, it is usually enough to change measures in $\R_{+}^{2}$, and then pass from $\R_{+}^{2}$ to $\Omega$ using the following result.

    \begin{theorem}[Theorem 2.8 of \cite{kenigWeighted}]\label{thm:thm2.8_kenig}
        Let $p\in (0,\infty)$ and $\nu \in A_{\infty}(\Lambda)$. Then $h \in H^p(\Omega, \nu)$ if and only if $h \circ \Phi \in H^p\left(\mathbb{R}_{+}^2, \Phi(\nu)\right)$, with equivalent norms.
    \end{theorem}
    
Indeed, an example of this pholosophy is Lemma 4.1 of \cite{CNO}, which can be deduced directly from Lemma 4.4 of \cite{CNO} using Theorem \ref{thm:thm2.8_kenig}. Since Corollary \ref{corol:analogoLema4.4CNO} generalizes Lemma 4.4 of \cite{CNO}, combining it with Theorem \ref{thm:thm2.8_kenig} we generalize Lemma 4.1 of \cite{CNO}.

\begin{corollary}\label{corol:analogoLema4.1CNO}
    Let $F$ be an analytic function in $\R_{+}^{2}$. Then
    \[
    (F\circ\Phi^{-1})\cdot (\Phi^{-1})'
    \in 
    H^p(\Omega,\nu)
    \iff 
    F\circ\Phi^{-1}\in H^p(\Omega,|(\Phi^{-1})'|^p\nu),
    \]
    with equivalent norms.
\end{corollary}

\begin{proof}
    Just notice that by Theorem \ref{thm:thm2.8_kenig},
    \[
    (F\circ\Phi^{-1})\cdot (\Phi^{-1})'
    \in 
    H^p(\Omega,\nu)
    \iff 
    F\cdot \frac{1}{\Phi'}
    \in 
    H^p(\R_{+}^{2},\Phi(\nu)),
    \]
    because $(\Phi^{-1})'\circ\Phi=\frac{1}{\Phi'\circ\Phi^{-1}}\circ\Phi=\frac{1}{\Phi'}$,  and
    \[
    F\circ\Phi^{-1}
    \in 
    H^p(\Omega,|(\Phi^{-1})'|^p\nu)
    \iff 
    F
    \in 
    H^p(\R_{+}^{2},|\Phi'|^{-p}\Phi(\nu)),
    \]
    because $\Phi(|(\Phi^{-1})'|^p\nu)=|\Phi'|^{-p}\Phi(\nu)$. Therefore, the results follow by Corollary \ref{corol:analogoLema4.4CNO}.
\end{proof}

\section*{Acknowledgements}

I want to thank my PhD supervisor, María Jesús Carro, for her careful reading and suggestions, that have greatly improved the original version and ideas of the manuscript. I am also grateful to Virginia Naibo for several comments about the final version.

\bibliographystyle{alpha}
\bibliography{references}

\end{document}